\documentclass[12pt]{amsart}

\usepackage{amsfonts}
\usepackage{amsmath}
\usepackage{amssymb}
\usepackage{latexsym}

\usepackage{setspace}

\newtheorem{theorem}{Theorem}
\theoremstyle{plain}

\newtheorem{definition}[theorem]{Definition}

\newtheorem{remark}[theorem]{Remark}

\numberwithin{equation}{section} \numberwithin{theorem}{section}

\newcommand\bel[1]{\begin{equation}\label{#1}}
\newcommand\ee{\end{equation}}

\newcommand\bE{{\mathbb{E}}}
\newcommand\bR{{\mathbb{R}}}

\def\cprime{$'$}

\begin{document}

%\today

\title{Estimating
Speed and Damping in the Stochastic Wave Equation}
\author{W. Liu}
\curraddr[W. Liu]{Department of Mathematics, USC\\
Los Angeles, CA 90089 USA\\
tel. (+1) 213 821 1480; fax: (+1) 213 740 2424}
\email[W. Liu]{liu5@usc.edu}
\author{S. V. Lototsky}
\curraddr[S. V. Lototsky]{Department of Mathematics, USC\\
Los Angeles, CA 90089 USA\\
tel. (+1) 213 740 2389; fax: (+1) 213 740 2424}
\email[S. V. Lototsky]{lototsky@usc.edu}
\urladdr{http://www-rcf.usc.edu/$\sim$lototsky}

\subjclass[2000]{Primary  62F12; Secondary 60G15, 60H15,
 60G30, 62M05}
\keywords{Cylindrical Brownian motion,
  Second-Order Stochastic Equations,
  Stochastic Hyperbolic Equations}

\thanks{The work was partially supported
by the NSF Grant DMS-0803378.}

 \begin{abstract}
A parameter estimation problem is considered for a one-dimensional stochastic
wave equation driven by additive space-time Gaussian white noise.
 The estimator is of spectral type and utilizes
 a finite number of the
spatial Fourier coefficients of the solution.
The asymptotic properties of the estimator are studied
as the number of the Fourier coefficients increases, while
the observation time and the noise intensity are fixed.
\end{abstract}

\maketitle

\section{Introduction}

Consider the stochastic wave equation
\begin{equation}
\label{SWE1-in}
\frac{\partial^2 u}{\partial t^2}=\theta_1\,\frac{\partial^2 u}{\partial x^2}
+\theta_2\,\frac{\partial u}{\partial t} + \dot{W}(t),\ 0<t<T,\ 0<x<\pi,
\end{equation}
with zero initial boundary conditions, driven by space-time white noise $\dot{W}$.
The solution of this equation can be written as a Fourier series
$$
u(t,x)=\sqrt{\frac{2}{\pi}}\,\sum_{k= 1}^{\infty}u_k(t)\,\sin(kx).
$$
The objective is to construct and investigate the maximum likelihood estimators of the
unknown numbers $\theta_1>0$ and $\theta_2\in \bR$, given
$\{u_1(t),\ldots, u_N(t)\}$, $t\in [0,T]$, the first $N$ Fourier coefficients of the
solution.

A similar problem for stochastic parabolic equations is relatively well studied, with the
first result announced in
the paper by Huebner, Khasminskii, and Rozovskii \cite{HKR}. While most of the existing papers
concentrate on estimating either a single parameter or a function of time,
estimation of several parameters in parabolic equations has also been studied
\cite{Hub1,Lot1}. The objective of the current paper is to extend some of the
results from \cite{Hub1} to stochastic hyperbolic equations. In both parabolic and
hyperbolic setting, the underlying assumption is that the solution $u=u(t,x)$ of the equation
can be measured at every point in time and space. Then the Fourier coefficients
$u_k$ of the solution can be computed and used to construct the estimator.

The main result of the paper is as follows.

\begin{theorem}
The (joint) maximum likelihood estimator of the parameters $\theta_1, \theta_2$
is strongly consistent and asymptotically normal as $N\to \infty$. The normalizing matrix
is diagonal, with the diagonal elements $N^{3/2}$ and $N^{1/2}$; these elements
specify the rate of convergence of the estimator to $\theta_1$ and $\theta_2$, respectively.
\end{theorem}

This theorem is proved in Section 3. In Section 2, we establish existence,
uniqueness, and regularity of the solution of \eqref{SWE1-in}.

Throughout the presentation below, we fix a
stochastic basis
 $$
 \mathbb{F}=(\Omega, \mathcal{F},\{\mathcal{F}_t\}_{t\geq 0},
\mathbb{P})
$$
with the usual assumptions
(completeness of $\mathcal{F}_0$ and right-continuity of
$\mathcal{F}_t$). We also assume that $\mathbb{F}$ is large
enough to support countably many independent standard Brownian
motions. For a random variable $\xi$, $\bE\xi$
denotes the expectation.
$\bR^n$ is an $n$-dimensional Euclidean space; $\mathcal{C}(A;B)$ is the
space of continuous functions from $A$ to $B$;
$\mathcal{N}(m,\sigma^2)$ is a Gaussian random variable with
mean $m$ and variance $\sigma^2$.

Finally, for the convenience of the reader, we recall that a cylindrical Brownian
motion $W=W(t)$, $t\geq 1$, over (or on) a Hilbert space $H$ is a linear mapping
$$
W: f\mapsto W_f(\cdot)
$$
 from
$H$ to the space of zero-mean Gaussian processes such that, for every
$f,g\in H$ and $t,s>0$,
\begin{equation}
\label{CBM}
\bE\big(W_f(t)W_g(s)\big)=\min(t,s)(f,g)_H.
\end{equation}
If $\{h_k,\ k\geq 1\}$ is an orthonormal basis in $H$ and $w_k,\ k\geq1,$ are
independent standard Brownian motions, then
\begin{equation}
\label{CBM1}
f\mapsto \sum_{k\geq 1} (f,h_k)_Hw_k(t)
\end{equation}
is a cylindrical Brownian motion. Thus, a cylindrical Brownian motion
$W$ is often represented by a generalized Fourier series
\begin{equation}
\label{CBM2}
W(t)=\sum_{k\geq 1} w_k(t) h_k,
\end{equation}
where $w_k=W_{h_k}$. The corresponding space-time white noise is then
$$
\dot{W}(t)=\sum_{k\geq 1} \dot{w}_k(t)h_k.
$$

\section{Stochastic Wave Equation}
Consider the  equation
\begin{equation}
\label{SWE1}
\frac{\partial^2 u}{\partial t^2}=a^2\frac{\partial^2 u}{\partial x^2}
-2b\frac{\partial u}{\partial t} + \dot{W}(t),\ 0<t<T,\ 0<x<\pi,
\end{equation}
where $W$ is a cylindrical Brownian motion over $L_2((0,\pi))$.
For simplicity, we assume
\begin{align}
\label{coef} &a \geq 1,\ \  2|b|\leq 1;\\
\label{IBC}
&u|_{t=0}=\frac{\partial u}{\partial t}\Bigg|_{t=0}=0,\ \
u|_{x=0}=u|_{x=\pi}=0;
\end{align}
see Remark \ref{rem:cond} below about relaxing these assumptions.
In physical models, $a>0$ represents the speed of the wave and $b$
characterizes damping (amplification, if $b<0$).

For $\gamma\in \bR$, define the Hilbert space $H^{\gamma}$ as the closure
of the set of smooth compactly supported functions on $(0,\pi)$
with respect to the norm
\begin{equation}
\label{SobSp}
\|f\|_{\gamma}=\left(\sum_{k\geq1} k^{2\gamma}f_k^2\right)^{1/2},\ {\rm where}\
f_k=\sqrt{\frac{2}{\pi}}\int_0^{\pi} f(x)\sin(kx)dx.
\end{equation}
Note that  each of the functions $\sin(kx)$ belongs to every $H^{\gamma}$, and
if $f$ is  twice continuously-differentiable on $(0,\pi)$ with $f(0)=f(\pi)=0$, then, after
two integrations by parts,
$|f_k|\leq k^{-2} \sup_{x\in (0,\pi)}|f''(x)|$, so that, in particular,  $f\in H^{1}$.
More generally, every  $f\in H^{\gamma}$ can be identified with a sequence $\{f_k,\ k\geq 1\}$ of real
numbers such that $\sum_{k\geq 1}k^{2\gamma}f_k^2<\infty$. Even though $f$ is a generalized
function when  $\gamma<0$, we will still occasionally write $f=f(x)$, keeping in mind a generalized
Fourier series
representation $f(x)=\sqrt{2/\pi}\,\sum_{k\geq 1}f_k\sin(kx)$.

Given $\gamma> 0$, $f\in H^{-\gamma}$ and
$g\in H^{\gamma}$, we define
$$
(f,g)=\sum_{k\geq 1} f_kg_k;
$$
if $f,g\in L_2((0,\pi))$, then
$$
(f,g)=\int_0^{\pi} f(x)g(x)dx.
$$
In other words, $(\cdot, \cdot)$ is the duality between $H^{\gamma}$ and $H^{-\gamma}$
relative to the inner product in $H^0=L_2((0,\pi))$; see
\cite[Section IV.1.10]{KPS}.

Equation \eqref{SWE1} is interpreted as a system of two first-order
It\^{o} equations
\begin{equation}
\label{SWE2}
du=vdt,\ dv=(a^2u_{xx}-2bv)dt+dW(t).
\end{equation}
More precisely, we have the following definition.

\begin{definition}
\label{def-sol}
An adapted process $u\in L_2\big(\Omega\times(0,T)\times (0,\pi)\big)$
is called a solution of \eqref{SWE1} if there
exists an adapted process $v$ such that
\begin{enumerate}
\item $v\in L_2\big(\Omega; L_2((0,T);H^{-1})\big)$;
\item For every twice continuously-differentiable on $(0,\pi)$ function $f=f(x)$
with $f(0)=f(\pi)=0$, the equalities
\begin{equation}
\label{SWE3}
\begin{split}
(u(t,\cdot),f)&=\int_0^t (v(t,\cdot),f)(s)ds,\\
(v(t,\cdot),f)&=\int_0^t \big(a^2(u(t,\cdot),f'')-2b(v(t,\cdot),f)\big)ds+W_f(t)
\end{split}
\end{equation}
hold for all $t\in [0,T]$ on the same set of probability one.
\end{enumerate}
\end{definition}

Here is the main result about existence and uniqueness of
solution of \eqref{SWE1}.

\begin{theorem}
\label{th:reg-sol}
Under assumptions \eqref{coef} and \eqref{IBC}, equation \eqref{SWE1}
has a unique solution and, for every $\gamma<1/2$,
\begin{equation}
\label{Reg-u}
u \in {L}_2\big(\Omega;L_2((0,T);{H}^{\gamma})\big);\ \
v\in  {L}_2\big(\Omega; L_2((0,T);{H}^{\gamma-1})\big).
\end{equation}
\end{theorem}

\begin{proof} While the result can be derived from the general theory
of stochastic hyperbolic equations (see, for example, Chow \cite[Theorem 6.8.4]{Chow-Spde}),
we present a different, and a more direct, proof. This proof will also help in the
construction and analysis of the estimators.

Take in \eqref{SWE3} $f(x)=\sqrt{2/\pi} \sin(kx)$ and write
$u_k(t)=(u(t,\cdot),f)$, $v_k(t)=(v(t,\cdot),f)$, $w_k=W_f$. Then
\begin{equation}
\label{Osc1}
u_k(t)=\int_0^tv_k(s)ds,\ v_k(t)=-a^2k^2\int_0^tu_k(s)ds-2b\int_0^t v_k(s)ds+w_k(t),
\end{equation}
or
\begin{equation}
\label{Osc2}
\ddot{u}_k(t)+2b\dot{u}_k(t)+a^2k^2u_k(t)=\dot{w}_k(t),
 \ \ u_k(0)=\dot{u}_k(0)=0.
\end{equation}
By assumption \eqref{coef},
\begin{equation}
\label{osc-condM}
a^2k^2>b^2
\end{equation}
 for all $k\geq 1$.
   Define
\begin{equation}
\label{ell-k}
\ell_k=\sqrt{a^2k^2-b^2}.
\end{equation}
Using the variation of parameters formula for the linear second-order
equation with constant coefficients, we conclude that the solution of \eqref{Osc1}
is
\begin{equation}
\label{Osc3}
\begin{split}
u_k(t)&=\frac{1}{\ell_k}\int_0^te^{-b(t-s)}\sin\big(\ell_k(t-s)\big)dw_k(s),\\
v_k(t)&=\frac{1}{\ell_k}\int_0^te^{-b(t-s)}\Big(
\ell_k\cos\big(\ell_k(t-s)\big) - b \sin\big(\ell_k(t-s)\big)\Big)dw_k(s).
\end{split}
\end{equation}
By direct computation, there exists a number $C=C(T,a,b)$ such that,
for all $t,s\in [0,T]$,
\begin{equation}
\label{Osc4}
\bE  u^2_k(t) \leq \ell_k^{-2} C(T)= \frac{C(T)}{a^2k^2-b^2},\ \ \
\bE  v_k^2(t) \leq {C(T)}.
\end{equation}
Then the Gaussian processes
\begin{equation}
\label{SOL}
u(t,x)=\sqrt{\frac{2}{\pi}}\sum_{k\geq 1} u_k(t)\sin(kx),\ \
v(t,x)= \sqrt{\frac{2}{\pi}}\sum_{k\geq 1} v_k(t)\sin(kx)
\end{equation}
satisfy \eqref{SWE3} and \eqref{Reg-u}. Uniqueness of the solution follows from
the completeness of the system $\{\sqrt{2/\pi} \sin(kx),\ k\geq 1\}$ in
$L_2((0,\pi))$.
\end{proof}

\begin{remark}
\label{rem:cond}  We can now comment on the significance of assumptions
\eqref{coef} and \eqref{IBC}. Assumption \eqref{coef} can be relaxed to
$a>0$, because we will still have $a^2k^2>b^2$ for all sufficiently
large $k$, and so representation formulas \eqref{Osc3} for the
solution of equation \eqref{Osc2} will continue to hold for all sufficiently large
$k$. In other words,  if $a>0$, then the free
motion (any solution of the homogeneous version of \eqref{Osc2}) is oscillatory
for all sufficiently large $k\geq 1;$ the oscillations are damped
  if $b>0$, harmonic if $b=0$, and amplified if $b>0$. This is also the reason
  to call $b$ the {\tt damping coefficient}, with an understanding that negative
  damping means amplification. Thus,  \eqref{coef} is only needed to simplify the
computations by ensuring that   equalities
 \eqref{Osc3} hold  {\em for all} $k\geq 1$.

Non-zero initial conditions, if sufficiently regular,
will not affect existence and regularity of the solution. Similarly, the analysis will
not change much for zero Neumann or other homogeneous boundary conditions.
\end{remark}

\section{Estimating the Coefficients}

In this section, we assume that the solution $u=u(t,x),\ v=v(t,x)$ of
equation \eqref{SWE1} is observed for all $0\leq t \leq T$ and
$x\in (0,\pi)$, and study the question of estimating the numbers
$a^2,b$ from these observations. It will be convenient to introduce the
notations
\begin{equation}
\label{not-coef}
\theta_1=a^2,\ \theta_2=-2b,
\end{equation}
so that \eqref{SWE1} becomes
\begin{equation}
\label{SWE-e1}
\frac{\partial^2 u}{\partial t^2}=\theta_1\frac{\partial^2 u}{\partial x^2}
+\theta_2\frac{\partial u}{\partial t} +
\dot{W}(t),\ t<0<T,\ 0<x<\pi.
\end{equation}
To simplify the presentation, we keep the assumptions \eqref{coef} and \eqref{IBC}.
By Theorem \ref{th:reg-sol}, the solution of \eqref{SWE-e1}
has a Fourier series expansion \eqref{Osc3}. We will construct the maximum likelihood
estimators of $\theta_1$ and $\theta_2$ using the observations of the
$2N$-dimensional process $\{u_k(t),\ v_k(t),\ k=1,\ldots,N,\ t\in [0,T]\}$ and
study the asymptotic properties of the estimators in the limit $N\to \infty$.
Note that both the amplitude of noise and the observation time are fixed.

By \eqref{Osc1},
\begin{equation}
\label{Osc1-e}
u_k(t)=\int_0^tv_k(s)ds,\ v_k(t)=-\theta_1k^2\int_0^tu_k(s)ds
+\theta_2\int_0^t v_k(s)ds+w_k(t).
\end{equation}
For each $k\geq 1$, the processes $u_k$, $v_k$, and $w_k$  generate measures
$\mathbf{P}^u_k$, $\mathbf{P}^v_k$, $\mathbf{P}^w_k$ in the
space $\mathcal{C}((0,T);\bR)$
of continuous, real-valued functions on $[0,T]$.
Since $u_k$ is a continuously-differentiable function, the
measures $\mathbf{P}^u_k$ and
$\mathbf{P}^w_k$ are mutually singular. On the other hand, we can write
\begin{equation}
\label{v-e}
dv_k(t)=F_k(v)dt+dw_k,
\end{equation}
where
$F_k(v)=-\theta_1k^2\int_0^tv_k(s)ds+\theta_2v_k(t)$ is a
non-anticipating functional of $v$.
Thus, the process $v$ is a process of diffusion type in the sense of
Liptser and Shiryaev \cite[Definition 4.2.7]{LSh1}. Further analysis shows
that the measure $\mathbf{P}^v_k$ is
absolutely continuous with respect to the measure $\mathbf{P}^w_k$, and
\begin{equation}
\label{density1-e}
\begin{split}
\frac{d\mathbf{P}^v_k}{d\mathbf{P}^w_k}(v_k)&=
\exp\Bigg(\int_0^T\big(-\theta_1k^2u_k(t)+\theta_2v_k(t)\big)dv_k(t)\\
& -\frac{1}{2}
\int_0^T\big(-\theta_1k^2u_k(t)+\theta_2v_k(t)\big)^2dt\Bigg);
\end{split}
\end{equation}
see \cite[Theorem 7.6]{LSh1}.
Since the processes $w_k$ are independent for different $k$, so are the processes $v_k$.
Therefore, the measure $\mathbf{P}^{v,N}$ generated in $\mathcal{C}((0,T);\bR^N)$
by the vector process
$\{v_k,\ k=1,\ldots,N\}$ is absolutely continuous with respect to the
 measure $\mathbf{P}^{w,N}$ generated in $\mathcal{C}((0,T);\bR^N)$
by the vector process
$\{w_k,\ k=1,\ldots,N\}$, and the density is
\begin{equation}
\label{density2-e}
\begin{split}
\frac{d\mathbf{P}^{v,N}}{d\mathbf{P}^{w,N}}(v_k)&=
\exp\Bigg(\sum_{k=1}^N\int_0^T\big(-\theta_1k^2u_k(t)+\theta_2v_k(t)\big)dv_k(t)\\
&-\frac{1}{2}
\sum_{k=1}^N\int_0^T\big(-\theta_1k^2u_k(t)+\theta_2v_k(t)\big)^2dt\Bigg);
\end{split}
\end{equation}
the corresponding log-likelihood ratio is
\begin{equation}
\label{log-lik}
\begin{split}
Z_N(\theta_1,\theta_2)&=\sum_{k=1}^N\Bigg(\int_0^T\big(-\theta_1k^2u_k(t)+\theta_2v_k(t)\big)dv_k(t)\\
&-\frac{1}{2}
\int_0^T\big(-\theta_1k^2u_k(t)+\theta_2v_k(t)\big)^2dt\Bigg).
\end{split}
\end{equation}
Introduce the following notations:
\begin{equation}
\label{notations}
\begin{split}
&J_{1,N}=\sum_{k=1}^N k^4\int_0^T u_k^2(t)dt,\ \ J_{2,N}=\sum_{k=1}^N\int_0^Tv_k^2(t)dt,\\
&J_{12,N}=\sum_{k=1}^Nk^2\int_0^Tu_k(t)v_k(t)dt;\\
&B_{1,N}=-\sum_{k=1}^N k^2 \int_0^Tu_k(t)dv_k(t),\ \xi_{1,N}=\sum_{k=1}^N k^2 \int_0^Tu_k(t)dw_k(t);\\
&B_{2,N}=\sum_{k=1}^N \int_0^Tv_k(t)dv_k(t),\ \xi_{2,N}=\sum_{k=1}^N \int_0^Tv_k(t)dw_k(t).
\end{split}
\end{equation}
Note that the numbers $J$ and $B$ are computable from the observations of $u_k$ and $v_k$, $k=1,
\ldots, N$, and also
\begin{align}
\label{ident-B}  B_{1,N}=\theta_1J_{1,N}&-\theta_2J_{12,N}-\xi_{1,N},\ \
B_{2,N}=-\theta_1J_{12,N}+\theta_2J_{2,N}+\xi_{2,N},\\
\label{ident-J} &J_{12,N}=\frac{1}{2}\sum_{k=1}^N k^2u_k^2(T).
\end{align}
We consider the problem of estimating simultaneously both
$\theta_1$ and $\theta_2$ from the observations
$$
\{u_k(t),\ v_k(t),\ k=1,\ldots, N,\ t\in [0,T].
$$
 The maximum likelihood estimators $\hat{\theta}_{1,N},
\ \hat{\theta}_{2,N}$ satisfy
$$
\frac{\partial Z_N(\theta_1,\theta_2)}{\partial \theta_1}
\Bigg|_{\theta_1=\hat{\theta}_{1,N},\theta_2=\hat{\theta}_{2,N}}=0
\quad {\rm and}\quad
\frac{\partial Z_N(\theta_1,\theta_2)}{\partial \theta_2}
\Bigg|_{\theta_1=\hat{\theta}_{1,N},\theta_2=\hat{\theta}_{2,N}}=0,
$$
 or, after solving the system of equations,
\begin{equation}
\label{est-th3}
\hat{\theta}_{1,N}=\frac{B_{1,N}J_{2,N}
+B_{2,N}J_{12,N}}{J_{1,N}J_{2,N}-J_{12,N}^2},\ \
\hat{\theta}_{2,N}=\frac{B_{1,N}J_{12,N}
+B_{2,N}J_{1,N}}{J_{1,N}J_{2,N}-J_{12,N}^2}.
\end{equation}

For $T>0$ and $\theta_2\in \bR$, define
\begin{equation}
\label{coef-th2}
C(\theta_2,T)=
\begin{cases}
\displaystyle \frac{e^{\theta_2T}-\theta_2T-1}{2\theta_2^2},&{\ \rm if}\ \theta_2\not=0;\\
\displaystyle  \frac{T^2}{4},& {\ \rm if}\ \theta_2=0.
\end{cases}
\end{equation}
Note that $C(\theta_2,T)>0$ for all $T>0$ and $\theta_2\in \bR$.

The following theorem describes the asymptotic behavior of the estimators
\eqref{est-th3}.

\begin{theorem}
\label{th:main}
Under assumptions \eqref{coef} and \eqref{IBC},
$$
\lim_{N\to \infty}\hat{\theta}_{1,N}^{(3)}=\theta_1,\ \
\lim_{N\to \infty}\hat{\theta}_{2,N}^{(3)}=\theta_2
$$
with probability one and
\begin{equation*}
\begin{split}
&\lim_{N\to \infty} N^{3/2}(\hat{\theta}_{1,N}-\theta_1)=
\mathcal{N}\left(0,\frac{3\theta_1}{C(\theta_2,T)}\right), \\
&\lim_{N\to \infty} N^{1/2}(\hat{\theta}_{2,N}-\theta_2)=
\mathcal{N}\left(0,\frac{1}{C(\theta_2,T)}\right)
\end{split}
\end{equation*}
in distribution.
\end{theorem}

\begin{proof}
Define
$$
D_N=\frac{J_{12,N}^2}{J_{1,N}J_{2,N}}.
$$
It follows from \eqref{ident-B} and \eqref{est-th3} that
\begin{equation}
\label{est-th3p}
\begin{split}
\hat{\theta}_{1,N}^{(3)}=\theta_1+\frac{1}{1-D_N}
\left(\frac{\xi_{1,N}}{J_{1,N}}+{\xi_{2,N}}\frac{J_{12,N}}{J_{1,N}J_{2,N}}
\right),\\
\hat{\theta}_{2,N}^{(3)}=\theta_2+\frac{1}{1-D_N}
\left(\frac{\xi_{2,N}}{J_{2,N}}+{\xi_{1,N}}\frac{J_{12,N}}{J_{1,N}J_{2,N}}
\right).
\end{split}
\end{equation}

By direct computations using \eqref{Osc3} (and keeping in mind \eqref{not-coef}),
\begin{equation}
\label{asympt-u}
\lim_{k\to \infty} k^2\bE\int_0^T u_k^2(t)dt=\frac{C(\theta_2,T)}{\theta_1},
\end{equation}
and
\begin{equation}
\label{asympt-u1}
\lim_{N\to \infty} N^{-3}\bE J_{1,N} = \frac{C(\theta_2,T)}{3\theta_1}.
\end{equation}
Since each $u_k$ is a Gaussian process,
$$
\sup_{k} k^4\bE\int_0^T u_k^4(t)dt  < \infty,
$$
and then the strong law of large numbers  implies
$$
\lim_{N\to \infty} \frac{J_{1,N}}{\bE J_{1,N}}=1,\
\lim_{N\to \infty}\frac{\xi_{1,N}}{\bE J_{1,N}}=0,
$$
both with probability one [apply the first theorem in
 Appendix, taking
$\xi_k=k^4\int_0^Tu_k^2dt$ and then, $\xi_k=k^2\int_0^Tu_k(t)dw_k(t)$].
 The central limit theorem  implies
$$
\lim_{N\to \infty} \frac{\xi_{1,N}}{\sqrt{\bE J_{1,N}}}=\mathcal{N}(0,1)
$$
in distribution [apply the second theorem in Appendix, taking $f_k(t)=k^2u_k(t)$].
Similarly,
\begin{equation}
\label{asympt-v}
\lim_{k\to \infty} \bE\int_0^T v_k^2(t)dt={C(\theta_2,T)},
\end{equation}
and
\begin{equation}
\label{asympt-v1}
\lim_{N\to \infty} N^{-1}\bE J_{2,N} = {C(\theta_2,T)}.
\end{equation}
Since each $v_k$ is a Gaussian process,
$$
\sup_{k} \bE\int_0^T v_k^4(t)dt  < \infty,
$$
and then the strong law of large numbers  implies
$$
\lim_{N\to \infty} \frac{J_{2,N}}{\bE J_{2,N}}=1,\
\lim_{N\to \infty}\frac{\xi_{2,N}}{\bE J_{2,N}}=0,
$$
both with probability one. The central limit theorem  implies
$$
\lim_{N\to \infty} \frac{\xi_{2,N}}{\sqrt{\bE J_{2,N}}}=\mathcal{N}(0,1)
$$
in distribution. Finally,
define
$$
\tilde{C}(\theta_2,T)=
\begin{cases}
\displaystyle \frac{e^{\theta_2T}-1}{2\theta_2},&{\ \rm if}\ \theta_2\not=0;\\
\displaystyle  \frac{T}{2},& {\ \rm if}\ \theta_2=0.
\end{cases}
$$
Then
 \eqref{ident-J} and \eqref{Osc3} imply
\begin{equation*}
\label{asympt-12}
\lim_{N\to \infty} N^{-1} \bE J_{12,N}= \frac{\tilde{C}(\theta_2,T)}{2\theta_1},\
\end{equation*}
and, by the strong law of large numbers,
$$
\lim_{N\to \infty} \frac{J_{12,N}}{\bE J_{12,N}}=1
$$
with probability one.
Then \eqref{asympt-u1} and \eqref{asympt-v1} imply
$$
\lim_{N\to \infty} D_N=0,\
\lim_{N\to \infty}\frac{J_{12,N}}{J_{2,N}}=\frac{\tilde{C}(\theta_2,T)}{2\theta_1{C}(\theta_2,T)},
$$
both  with
probability one. The conclusions of the theorem now follow.
\end{proof}

\section{Acknowledgement} The work of SVL
 was partially supported
by the NSF Grant DMS-0803378.

%\bibliographystyle{plain}
%\bibliography{StatEst}

\begin{thebibliography}{1}

\bibitem{Chow-Spde}
P.-L. Chow.
\newblock {\em Stochastic partial differential equations}.
\newblock Chapman \& Hall/CRC, Boca Raton, FL, 2007.

\bibitem{Hub1}
M.~Huebner.
\newblock A characterization of asymptotic behaviour of maximum likelihood
  estimators for stochastic {PDE's}.
\newblock {\em Math. Methods Statist.}, 6(4):395--415, 1997.

\bibitem{HKR}
M.~Huebner, R.~Z. Khas{\cprime}minski{\u\i}, and B.~L. Rozovskii.
\newblock Two examples of parameter estimation.
\newblock In S.~Cambanis, J.~K. Ghosh, R.~L. Karandikar, and P.~K. Sen,
  editors, {\em Stochastic Processes: A volume in honor of {G}. {K}allianpur},
  pages 149--160. Springer, New York, 1992.

\bibitem{JSh}
J.~Jacod and A.~N. Shiryaev.
\newblock {\em Limit theorems for stochastic processes, 2nd Ed.}, volume 288 of
  {\em Grundlehren der Mathematischen Wissenschaften}.
\newblock Springer, 2003.

\bibitem{KPS}
S.~G. Kre{\u\i}n, Yu.~{\=I}. Petun{\={\i}}n, and E.~M. Sem{\"e}nov.
\newblock {\em Interpolation of linear operators}, volume~54 of {\em
  Translations of Mathematical Monographs}.
\newblock American Mathematical Society, Providence, R.I., 1982.

\bibitem{LSh3}
R.~Sh. Liptser and A.~N. Shiryaev.
\newblock {\em Theory of martingales}, volume~49 of {\em Mathematics and its
  Applications (Soviet Series)}.
\newblock Kluwer Academic Publishers, Dordrecht, 1989.

\bibitem{LSh1}
R.~Sh. Liptser and A.~N. Shiryaev.
\newblock {\em Statistics of random processes, {I}: {G}eneral Theory, 2nd Ed.},
  volume~5 of {\em Applications of Mathematics}.
\newblock Springer, 2001.

\bibitem{Lot1}
S.~V. Lototsky.
\newblock Parameter estimation for stochastic parabolic equations: asymptotic
  properties of a two-dimensional projection-based estimator.
\newblock {\em Stat. Inference Stoch. Process.}, 6(1):65--87, 2003.

\bibitem{Shir}
A.~N. Shiryaev.
\newblock {\em Probability, 2nd Ed.}, volume~95 of {\em Graduate Texts in
  Mathematics}.
\newblock Springer, 1996.

\end{thebibliography}

%%%%%%%%%%%%%%%%%%%%%%%%%%%%%%%%%%%%%%%%%%%%
\def\cprime{$'$} \def\cprime{$'$} \def\cprime{$'$} \def\cprime{$'$}
  \def\cprime{$'$} \def\cprime{$'$}

%%%%%%%%%%%%%%%%%%%%%%%%%%%%%%%%%%%%%%%%%%%%%

\section*{Appendix}
Below, we formulate the strong law of large numbers and the central
limit theorem used in the proof of Theorem \ref{th:main}.

\begin{theorem}[Strong Law of Large Numbers]
Let $\xi_k, \ k\geq 1,$ be independent random variables with the following
properties:
\begin{itemize}
\item $\bE \xi_k=0$,\ $\bE \xi_k^2>0$,
\item There exist real numbers $c>0$ and $ \alpha\geq -1$ such that
$$
\lim_{k\to \infty} k^{-\alpha}\bE \xi_k^2=c.
$$
\end{itemize}
Then, with probability one,
$$
\lim_{N\to \infty} \frac{\sum_{k=1}^N\xi_k}{\sum_{k=1}^N \bE \xi_k^2}=0.
$$
If, in addition, $\bE \xi_k^4 \leq c_1\Big(\bE\xi_k^2\Big)^2$ for all $k\geq 1$,
with $c_1>0$ independent of $k$, then, also with probability one,
$$
\lim_{N\to \infty} \frac{\sum_{k=1}^N \xi_k^2}{\sum_{k= 1}^N \bE \xi_k^2}=1.
$$
\end{theorem}
\begin{proof} This is a particular case of Kolmogorov's strong law of large numbers;
see, for example, Shiryaev \cite[Theorem IV.3.2]{Shir}.
\end{proof}

\begin{theorem}[Central Limit Theorem]
Let $w_k=w_k(t)$ be independent standard Brownian motions and
let $f_k=f_k(t)$ be adapted, continuous,
 square-integrable processes such that
$$
\lim_{N\to \infty} \frac{\sum_{k=1}^N \int_0^Tf_k^2(t)dt}
{\sum_{k=1}^N\bE\int_0^T f_k^2(t)dt}=1
$$
in probability. Then
$$
\lim_{N\to \infty}
\frac{\sum_{k=1}^N \int_0^Tf_k(t)dw_k(t)}
{\left(\sum_{k=1}^N\bE\int_0^T f_k^2(t)dt\right)^{1/2}}=
\mathcal{N}(0,1)
$$
in distribution.
\end{theorem}

\begin{proof}
This is a particular case of a martingale limit theorem; see,
for example Jacod and Shiryaev \cite[Theorem VIII.4.17]{JSh}
or Liptser and Shiryaev \cite[Theorem 5.5.4(II)]{LSh3}.
\end{proof}

\end{document}